\documentclass[12pt,showkeys]{amsart}
\usepackage{color}
\usepackage{graphicx}
\usepackage{subfigure}
\usepackage{amssymb}
\usepackage{amsmath}
\usepackage{multirow}
\usepackage{enumerate}
\usepackage{epstopdf}
\usepackage{cite,stmaryrd,txfonts}
\usepackage{tikz}

\input{undertilde}
\usetikzlibrary{snakes}

\usepackage{verbatim}

\usepackage{epic}%% package for dash line \dashline

\usepackage{empheq}

\newcommand{\bs}{\boldsymbol}
\newcommand{\mrm}{\mathrm}

\def\cequiv{\raisebox{-1.5mm}{$\;\stackrel{\raisebox{-3.9mm}{=}}{{\sim}}\;$}}

\def\uw{\undertilde{w}}
\def\uv{\undertilde{v}}

\def\usigma{\undertilde{\sigma}}
\def\utau{\undertilde{\tau}}
\def\uphi{\undertilde{\varphi}}
\def\upsi{\undertilde{\psi}}
\def\ueta{\undertilde{\eta}}

\def\ugamma{\undertilde{\gamma}}

\def\uG{\undertilde{G}}

\def\uH{\undertilde{H}}

\def\uk{\undertilde{k}}

\def\ux{\undertilde{x}}

\def\up{\undertilde{p}}
\def\uP{\undertilde{P}}

\def\uS{\undertilde{S}}
\def\rot{{\rm rot}}
\def\curl{{\rm curl}}
\def\dv{{\rm div}}

\newtheorem{theorem}{Theorem}
\newtheorem{remark}[theorem]{Remark}

\newtheorem{lemma}[theorem]{Lemma}

\newcounter{mnote}
\let\oldmarginpar\marginpar
\renewcommand\marginpar[1]{\-\oldmarginpar[\raggedleft\footnotesize #1]%
  {\raggedright\footnotesize #1}}

\def\llt{\llbracket}
\def\rrt{\rrbracket}

\begin{document}
%\xiaoerhao

\def\uPnsk{\uP{}^{\mathbf{n},s}_k}
\def\uPnskx{\uP{}^{\mathbf{n},s}_{\uk+k\ux}}

\title[Optimal cubic element for biharmonic equation]{An optimal piecewise cubic nonconforming finite element scheme for the planar biharmonic equation on general triangulations}

\author{Shuo Zhang}
\address{LSEC, Institute of Computational Mathematics and Scientific/Engineering Computing, Academy of Mathematics and System Sciences, Chinese Academy of Sciences, Beijing 100190, People's Republic of China}
\email{szhang@lsec.cc.ac.cn}
\thanks{The author is supported by NCMIS of CAS and NSFC Grants Nos 11471026 and 11871465.}

\subjclass[2000]{Primary 65N30, 35Q60, 76E25, 76W05.}

\keywords{biharmonic equation, discretized Stokes complex, optimal finite element scheme}

\begin{abstract}
This paper presents a nonconforming finite element scheme for the planar biharmonic equation which applis piecewise cubic polynomials ($P_3$) and possesses $\mathcal{O}(h^2)$ convergence rate in energy norm on general shape-regular triangulations. Both Dirichlet and Navier type boundary value problems are studied. The basis for the scheme is a piecewise cubic polynomial space, which can approximate the $H^4$ functions with $\mathcal{O}(h^2)$ accuracy in broken $H^2$ norm. Besides, an equivalence $(\nabla_h^2\ \cdot,\nabla_h^2\ \cdot)=(\Delta_h\ \cdot,\Delta_h\ \cdot)$, which is usually not true for nonconforming finite element spaces, is proved on the newly designed spaces.  

The finite element space does not correspond to a finite element defined with Ciarlet's triple; however, a set of locally supported basis functions of the finite element space is still figured out. The notion of the finite element Stokes complex plays an important role in the analysis and also the construction of the basis functions. 
\end{abstract}

\maketitle

%\tableofcontents

%
%
%
\section{Introduction}
%\label{sec:into}
In order to obtain a simpler interior structure, in the study of the numerical analysis of partial differential equations, lower-degree polynomials are often expected to be used with respect to the same convergence rate. Finite element schemes with polynomials of degrees not higher than $k$ for $H^{\bf m}$ problem that possess convergence rates of $\mathcal{O}(h^{\bf k+1-m})$ in energy norm for solutions in $H^{\bf k+1}$ are called {\bf optimal}. According to \cite{Lin.Q;Xie.H;Xu.J2014}, this illustrates both the highest accuracy with respect to certain degree of polynomials and the smallest shape function space with respect to certain convergence rate, and is a critical characteristic for the finite element methodology.  Motivated by the fundamental problem aforementioned, this paper concerns the optimal finite element scheme for the biharmonic equation with piecewise cubic polynomials on general triangulations.

\paragraph{\bf A brief review of relevant works} Papers on optimal schemes can be found focusing mainly on {\it low-order} problems. For the lowest-differentiation-order ($H^1$) elliptic problems, the standard Lagrangian elements can yield optimal approximation on the simplicial grids of an arbitrary dimension. Further, the optimal nonconforming element spaces of $k$-th degrees are also constructed, c.f., e.g., \cite{Crouzeix.M;Raviart.P1973}, \cite{Fortin.M;Soulie.M1983}, and \cite{Crouzeix.M;Falk.R1989} for the cases $k=1$, $k=2$, and $k=3$, respectively, and \cite{Baran.A;Stoyan.G2007} for general $k$. For higher-differentiation-order ($H^m$, $m>1$) elliptic problems, minimal-degree approximations have been studied with the lowest accuracy order. Specifically, when the subdivision comprises simplexes, a systematic family of nonconforming finite elements has been proposed by \cite{Wang.M;Xu.J2013} for $H^m$ elliptic partial differential equations in $\mathbb{R}^n$ for any $n\geqslant m$ with polynomials with degree $m$. Besides, the constructions of finite element functions that do not depend on cell-by-cell definitions can be found in \cite{Park.C;Sheen.D2003,Hu.J;Shi.Z2005,Zhang.S2018ima}, wherein minimal-degree finite element spaces are defined on general quadrilateral grids for $H^1$ and $H^2$ problems. In contrast to these existing lowest order researches, the construction of higher-accuracy-order optimal schemes for higher-differentiation-order problems is complicated, even for the planar biharmonic problem.

Conforming finite elements for biharmonic equation requires the $C^1$ continuity assumption. It is well-known that with polynomials of degrees $k\geqslant 5$, spaces of $\mathcal{C}^1$ continuous piecewise polynomials can be constructed with local basis. Moreover, these spaces perform optimal approximations of $H^2$ functions with sufficient smoothness\cite{Argyris.J;Fried.I;Scharpf.D1968,Zenisek.A1974,Morgan.J;Scott.R1975,Zenisek.A1970,deBoor.C;Hollig.K1988}. With polynomials of degrees $2\leqslant k\leqslant 4$, spaces of $\mathcal{C}^1$ continuous piecewise polynomials can be shown to provide optimal approximation when the triangulation is of some special structures, such as the Powell--Sabin and Powell--Sabin--Heindl triangulations\cite{Powell.M;Sabin.M1977,Heindl.G1979,Powell.M1976}, criss-cross triangulations \cite{Zhang.Shangyou2008}, Hsieh--Clough--Tocher triangulation\cite{Clough.R;Tocher.J1965}, and Sander--Veubeke triangulation\cite{Sander.G1964,deVeubeke.BF1968}. The conditions on the grids can be relaxed, but they are generally required on at least some part of the triangulation\cite{Nurnberger.G;Zeilfelder.F2004,Nurnberger.G;Schumaker.L;Zeilfelder.F2004,Chui.C;Hecklin.G;Nurnberger.G;Zeilfelder.F2008}. On general triangulations, as is shown in \cite{deBoor.C;Jia.R1993}, optimal approximation cannot be obtained with $\mathcal{C}^1$ continuous piecewise polynomials of degree $k<5$. It is illustrated in \cite{Alfeld.P;Piper.B;Schumaker.L1987} that not all the basis functions can be determined locally on general grids. We would particularly recall a counterexample that, as studied in \cite{Boor.C;DeVore.R1983,Boor.C;Hollig.K1983,Babuska.I;Suri.M1992}, the $\mathcal{C}^1-P_3$ scheme is only $\mathcal{O}(h)$ order convergent in energy norm on a triangulation obtained by subdividing a rectangular domain with three groups of parallel lines (cf. Figure \ref{fig:counter}), which is even though one of the simplest and most regular triangulations.

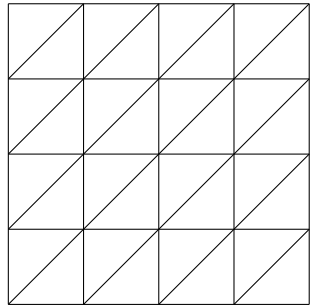
\begin{figure}[htbp]
\label{fig:counter}
\centering
\begin{tikzpicture}
	\draw(-2,0)--(2,0);
	\draw(-2,2)--(2,2);
	\draw(-2,-2)--(2,-2);
	\draw(-2,-1)--(2,-1);
	\draw(-2,1)--(2,1);
	
	\draw(-2,-2)--(-2,2);
	\draw(2,-2)--(2,2);
	\draw(-1,-2)--(-1,2);
	\draw(0,-2)--(0,2);
	\draw(1,-2)--(1,2);
	
	\draw(-1,2)--(-2,1);
	\draw(0,2)--(-2,0);
	\draw(1,2)--(-2,-1);
	\draw(2,2)--(-2,-2);
	\draw(2,1)--(-1,-2);
	\draw(2,0)--(0,-2);
	\draw(2,-1)--(1,-2);
	
\end{tikzpicture}
\caption{On the triangulations of this profile, optimal finite element scheme for biharmonic equation with piecewise cubic polynomials, conforming or nonconforming, is NOT yet known. }
\end{figure}

In contrast, a nonconforming finite element methodology, namely, the Morley element \cite{Morley.L1968}, which uses piecewise quadratic polynomials with a convergence rate of $\mathcal{O}(h)$, was shown to perform optimally for $k=2$. However, to the best of our knowledge, optimal piecewise cubic or quartic finite element schemes (either conforming or nonconforming) for a planar biharmonic equation with $\mathcal{O}(h^2)$ or $\mathcal{O}(h^3)$ convergence rate have not been discovered. We remark that several $\mathcal{O}(h^2)$ ordered finite element methods are designed with piecewise cubic polynomials enriched with higher-degree bubbles (e.g., \cite{Guzman.J;Leykekhman.D;Neilan.M2012,Wang.M;Zu.P;Zhang.S2012}). As the degrees of the functions exceed three, these methods are not considered optimal here. For a biharmonic problem in higher dimensions and other problems with higher orders, bigger difficulties can be expected.

\paragraph{\bf Main results of the present paper} In this paper, a space $B^3_h$ is constructed with piecewise cubic polynomials, whose subspaces $B^3_{ht}$ and $B^3_{h0}$ are proved to provide optimal approximation of $H^2\cap H^1_0$ and $H^2_0$, respectively. Finite element schemes that apply the two subspaces to the biharmonic equation with Navier and Dirichlet boundary conditions, respectively, are nonconforming, but the consistency errors are both  of $\mathcal{O}(h^2)$ order. Thus the finite element schemes are optimal, and the optimality can be proved on any shape regular grids on both convex and nonconvex polygonal domains. 

Further, for any two functions $w_h,v_h\in B^3_{ht}$, it can be proved that $(\nabla_h^2w_h,\nabla_h^2v_h)=(\Delta_hw_h,\Delta_hv_h)$, which is seldom true for nonconforming finite elements. This property makes the finite element spaces fit for the discretization of biLaplacian operator $\Delta\mathcal{A}\Delta$ with varying coefficient $\mathcal{A}$. 

Two approaches of implementing the schemes are suggested. One is to figure out their local basis functions: the finite element scheme does not correspond to a finite element in Ciarlet's triple; but the finite element spaces do possess local basis functions that each is supported in the patch of a vertex or the patch of an edge. The other is to decompose the finite element scheme to three decoupled subproblems, which are either a Poisson system or a Stokes system, to solve sequentially. Note that the optimal solvers for discrete Poisson system and Stokes problem have been very well developed, and the latter approach suggests indeed a method to solve the finite element problem with optimal cost.

\paragraph{\bf Main technical ingredient of the present paper} For the nonconforming finite element space $B_h^3$, to control the consistency error, sufficient restrictions on the interfacial continuity have to be imposed across the edges of the cells. However, {\it the constraints on the continuity are overdetermined in comparison to local shape functions}; hence, the {\it global finite element space} do not correspond to a {\it local finite element} defined with Ciarlet's triple. The functions can be viewed as {\it nonsmooth splines}.  Consequently, several challenges arise in both theoretical analysis and practical implementation, even on counting the dimension of the space. To avoid these challenges, in this paper, indirect methods are adopted; namely, the construction and utilization of discretized Stokes complexes constitute the bulk of the task in the construction of the space and schemes. This {\bf indirect} approach is viewed as the main ingredient of the paper.  

Discretized Stokes complexes are finite element analogs of the 2D Stokes complexes (or the de Rham complex with enhanced regularity), which read corresponding to the boundary condition: 
\begin{equation}
\begin{array}{ccccccccc}
0 & \xrightarrow{\bs{inclusion}} & H^2_0 & \xrightarrow{\bs{\mrm{\nabla}}} & (H^1_0)^2 & \xrightarrow{\mrm{rot}} & L^2_0  & \xrightarrow{\bs{\int\cdot}} & 0.
\end{array}
\end{equation}
and
\begin{equation}
\begin{array}{ccccccccc}
0 & \xrightarrow{\bs{inclusion}}  & H^2\cap H^1_0 & \xrightarrow{\bs{\mrm{\nabla}}} & (H^1)^2\cap H_0(\rot) & \xrightarrow{\mrm{rot}} & L^2_0  & \xrightarrow{\bs{\int\cdot}} & 0.
\end{array}
\end{equation}
In the complex, the combination of the successive two operators vanish, and the kernel of the latter one is exactly the range of the former one. The finite element complexes have been widely used for stability analysis (c.f.\cite{Arnold.D;Falk.R;Winther.R2006}), and, in this paper, the important role they play is four-folded:
\begin{enumerate}
\item It is used for approximation analysis. We construct two discretized Stokes complexes that start with finite element spaces $B^3_{h0}$ and $B^3_{ht}$, respectively, for $H^2$ and estimates the approximation error of $B^3_{h0}$($B^3_{ht}$) by estimating the discretization error of the auxiliary finite element discretization of the Stokes problem. This way, we prove the $\mathcal{O}(h^2)$ approximation accuracy of $B^3_{h0}$($B^3_{ht}$) in energy norm for $H^4$ functions. Moreover, the proof does not require a convexity assumption on the domain.
\item Different from existing nonconforming finite elements such as the Morley element, for $w_h,v_h\in B^3_{ht}$, $(\nabla_h^2w_h,\nabla_h^2v_h)=(\Delta_hw_h,\Delta_hv_h)$, the operations done cell by cell. This makes the finite element suitable for, e.g., $\Delta\mathcal{A}\Delta$ with varying coefficient $\mathcal{A}$; see \cite{Xi.Y;Ji.X;Zhang.S2019arxiv} for a practical application. This property is, once again, proved by the aid of the discretized Stokes complex. 
\item Further, though the finite element space does not correspond to a finite element defined in Ciarlet's triple, the finite element spaces do admit a set of basis functions, each of which is supported in a patch of a vertex or a patch of an edge. Again, the discretized Stokes complexes play crucial roles in proving the existence of the locally supported basis functions. 
\item Finally, we remark, beyond bringing ease in constructing and analyzing the schemes, the discretized Stokes complex is also helpful to the implementation and numerical solution of the systems by the aid of the discretized Poisson and discretized Stokes systems; we also refer to  \cite{Xu.J1996as,Hiptmair.R;Xu.J2007,Xu.J2010icm,Zhang.S;Xu.J2014,Ruesten.T;Winther.R1992,Grasedyck.L;Wang.L;Xu.J2016,Feng.C;Zhang.S2016} for relevant discussions.  
\end{enumerate}

\paragraph{\bf Bibliographic remark} This paper collects some original results from the unpublished arXiv preprints 1805.03851(\cite{Zhang.S2018}) authored by the same author as the present paper.

\paragraph{\bf Organization of the paper} The remaining of the paper is organized as follows. Section \ref{sec:fes} presents some finite element spaces and finite element complexes. Section \ref{sec:cubic}  presents two optimal nonconforming finite element schemes, including the construction, theoretical analysis, for the two kinds of boundary value problems, respectively. Two approaches of implementation are given in Section \ref{sec:imp}. Finally, in Section \ref{sec:conc}, some conclusions and further discussions are given.

\section{Finite element spaces and finite element complexes}
\label{sec:fes}

\subsection{Preliminaries}

In what follows, we use $\Omega$ to denote a simply connected polygonal domain, and $\nabla$, $\curl$, $\dv$, $\rot$, and $\nabla^2$ to denote the gradient operator, curl operator, divergence operator, rot operator, and Hessian operator, respectively. As usual, we use $H^2(\Omega)$, $H^2_0(\Omega)$, $H^1(\Omega)$, $H^1_0(\Omega)$, $H(\rot,\Omega)$, $H_0(\rot,\Omega)$, and $L^2(\Omega)$ to denote certain Sobolev spaces, and specifically, denote $\displaystyle L^2_0(\Omega):=\{w\in L^2(\Omega):\int_\Omega w dx=0\}$, $\undertilde{H}{}^1_0(\Omega):=(H^1_0(\Omega))^2$, and $\uH{}^1_t(\Omega)=(H^1(\Omega))^2\cap H_0(\rot,\Omega)$. Furthermore, we denoted vector-valued quantities by $``\undertilde{~}"$, while $\uv{}^1$ and $\uv{}^2$ denote the two components of the function $\uv$. We use $(\cdot,\cdot)$ to represent $L^2$ inner product, and $\langle\cdot,\cdot\rangle$ to denote the duality between a space and its dual. Without ambiguity, we use the same notation $\langle\cdot,\cdot\rangle$ for different dualities, and it can occasionally be treated as $L^2$ inner product for certain functions. We use the subscript $``\cdot_h"$ to denote the dependence on triangulation. In particular, an operator with the subscript $``\cdot_h"$ indicates that the operation is performed cell-by-cell. Finally, $\cequiv$ denotes equality up to a constant. The hidden constants depend on the domain, and when triangulation is involved, they also depend on the shape regularity of the triangulation, but they do not depend on $h$ or any other mesh parameter.

Let $\mathcal{T}_h$ be a shape-regular triangular subdivision of $\Omega$ with mesh size $h$, such that $\overline\Omega=\cup_{T\in\mathcal{T}_h}\overline T$. Denote by $\mathcal{E}_h$, $\mathcal{E}_h^i$, $\mathcal{E}_h^b$, $\mathcal{X}_h$, $\mathcal{X}_h^i$, $\mathcal{X}_h^b$ and $\mathcal{X}_h^c$ the set of edges, interior edges, boundary edges, vertices, interior vertices, boundary vertices and corners, respectively. For any edge $e\in\mathcal{E}_h$, denote by $\mathbf{n}_e$ and $\mathbf{t}_e$ the unit normal and tangential vectors of $e$, respectively, and denote by $\llbracket\cdot\rrbracket_e$ the jump of a given function across $e$; if particularly $e\in\mathcal{E}_h^b$, $\llbracket\cdot\rrbracket_e$ stands for the evaluation of the function on $e$. The subscript ${\cdot}_e$ can be dropped when there is no ambiguity brought in. 

Denote 
$$
\mathcal{X}_h^{b,+1}:=\{a\in\mathcal{X}_h^i,\ a\ \mbox{is\ connected\ to}\ \mathcal{X}_h^b\ by\ e\in\mathcal{E}_h^i\},\ \mbox{and}\ \ \mathcal{X}_h^{i,-1}:=\mathcal{X}_h^i\setminus\mathcal{X}_h^{b,+1};
$$ 
further, denote with $\mathcal{X}^{i,-(k-1)}_h\neq\emptyset$,
$$
\mathcal{X}_h^{b,+k}:=\{a\in\mathcal{X}_h^{i,-(k-1)},\ a\ \mbox{is\ connected\ to}\ \mathcal{X}_h^{b,+(k-1)}\ by\ e\in\mathcal{E}_h^i\}, \ \mbox{and}\ \ \mathcal{X}_h^{i,-k}:=\mathcal{X}_h^{i,-(k-1)}\setminus\mathcal{X}_h^{b,+k}.
$$
The smallest $k$ such that $\mathcal{X}_h^{i,-(k-1)}=\mathcal{X}_h^{b,+k}$ is called the number of levels of the triangulation.

For a triangle $T$, we use $P_k(T)$ to denote the set of polynomials on $K$ of degrees not higher than $k$. In a similar manner, $P_k(e)$ is defined on an edge $e$. We define $\uP{}_k(T)=P_k(T)^2$ and similarly is $\uP{}_k(e)$ defined. We use $a_i$, $i=1,2,3$ for the vertices of $T$ in an anticlockwise order, $e_i$, $i=1,2,3$ for the edges opposite to $a_i$, respectively, and $\lambda_i$, $i=1,2,3$ the barycentric coordinates. 

Also, we denote basic finite element spaces by
\begin{itemize}
\item $\mathcal{L}^k_h:=\{w\in H^1(\Omega):w|_T\in P_k(T),\ \forall\,T\in\mathcal{T}_h\}$, $\mathcal{L}^k_{h0}:=\mathcal{L}^k\cap H^1_0(\Omega)$, $k\geqslant 1$;
\item $\mathbb{P}^{k}_h:=\{w\in L^2(\Omega):w|_T\in P_{k}(T)\}
$, $\mathbb{P}^{k}_{h0}:=\mathbb{P}^k_h\cap L^2_0(\Omega)$, $k\geqslant 0$;
\item $\uS{}_h^k:=(\mathbb{P}^k_h)^2\cap \uH{}^1(\Omega)$, $k\geqslant 1$, $\uS{}^k_{ht}:=\uS{}^k_h\cap H_0(\rot,\Omega)$ and $\uS{}^k_{h0}:=\uS{}^k_h\cap\uH{}^1_0(\Omega)$;
\item $\uG{}_h^k:=\{\uv\in (\mathbb{P}^k_h)^2:\int_ep_e\llt\uv^j\rrt=0,\ \forall\, p_e\in P_{k-1}(e),\ \forall\, e\in \mathcal{E}_h^i,\ j=1,2\}$, $k\geqslant 1$, $\uG{}^k_{ht}:=\{\uv\in\uG{}^k_h:\int_ep_e\uv\cdot\mathbf{t}_e=0,\ \forall\,e\in \mathcal{E}_h^b\ \mbox{and}\ p_e\in P_{k-1}(e)\}$,  and $\uG{}^k_{h0}:=\{\uv\in\uG{}^k_h:\int_ep_e\uv^j=0,\ \forall\,e\in \mathcal{E}_h^b\ \mbox{and}\ p_e\in P_{k-1}(e),\ j=1,2\}.$ 
\end{itemize} 
Namely, $\uS{}^k_h$ consists of continuous functions, and $\uG{}^k_h$ consists of $(k-1)^{\rm th}$ order moment-continuous functions. Particularly, the space $\uG{}^2_h$ corresponds to the famous Fortin-Soulie element \cite{Fortin.M;Soulie.M1983}. The following stability result is well-known.

\begin{lemma}\label{lem:pGL}\cite{Fortin.M;Soulie.M1983}
There exists a generic constant $C$ depending on the domain and the regularity of the grid, such that 
\begin{equation}\label{eq:inf-supGL}
\sup_{\uv{}_h\in \uG{}^2_{h0},\|\nabla_h\uv{}_h\|_{0,\Omega}=1}(\dv_h\uv{}_h,q_h)\geqslant C\|q_h\|_{0,\Omega},\ \ \forall\,q_h\in\mathbb{P}^{1}_{h0}.
\end{equation}
\end{lemma}

\begin{remark}
By the symmetry between the two components of $\uH{}^1(\Omega)$, Lemma \ref{lem:pGL} remains true when ``$\dv{}_h$" is replaced by ``$\rot_h$."
\end{remark}

Denote $\undertilde{\mathcal{B}}{}^2_{h0}:=\{\undertilde{\phi}{}_h:(\undertilde{\phi}{}_h|_T)^j\in {\rm span}\{(\lambda_1^2+\lambda_2^2+\lambda_3^2)-2/3\}, \ j=1,2,\ \forall\,T\in\mathcal{T}_h\}$ and evidently the first order moments of $\undertilde{\phi}{}_h$  vanish along any edge of $\mathcal{T}_h$. Then $\uG{}^2_{h0}=\uS{}^2_{h0}\oplus \undertilde{\mathcal{B}}{}^2_{h0}$ (c.f. \cite{Fortin.M;Soulie.M1983}). Also $\uG{}^2_{ht}=\uS{}^2_{ht}\oplus \undertilde{\mathcal{B}}{}^2_{h0}$. Note that, as is known, $\uG{}^2_{h}=\uS{}^2_{h}+ \undertilde{\mathcal{B}}{}^2_{h0}$ is not a direct sum.  The decomposition can be generalized to even $k$ (c.f. \cite{Baran.A;Stoyan.G2007}). 
\begin{lemma}\label{lem:h1=d+c}
For any $\uw{}_h,\uv{}_h\in\uG{}^2_{ht}$, it holds that 
\begin{equation}\label{eq:h1=d+c}
(\nabla_h\uw{}_h,\nabla_h\uv{}_h)=(\dv{}_h\uw{}_h,\dv_h\uv{}_h)+(\rot_h\uw{}_h,\rot_h\uv{}_h).
\end{equation}
\end{lemma}
\begin{proof}
Firstly,  \eqref{eq:h1=d+c} holds for any $\uw{}_h,\uv{}_h\in\uS{}^2_{ht}\subset \uH{}^1_t(\Omega)$. Secondly, \eqref{eq:h1=d+c} holds for any $\uw{}_h\in \uG{}^2_{ht}$ and $\uv{}_h\in\mathcal{B}_{h0}$; actually, for any $K\in\mathcal{T}_h$,
\begin{multline*}
\int_K\nabla \uw{}_h:\nabla \uv{}_h=-\int_K\Delta \uw{}_h\uv{}_h+\int_{\partial K}\partial_{\bf n}\uw{}_h\uv{}_h
= -\int_K\Delta \uw{}_h\uv{}_h= -\int_K(\nabla\dv+\curl\rot)\uw{}_h\uv{}_h 
\\
= -\int_K(\nabla\dv+\curl\rot)\uw{}_h\uv{}_h+\int_{\partial K}(\dv\uw{}_h\uv{}_h\cdot\mathbf{n}+\rot\uw{}_h\uv{}_h\cdot\mathbf{t})
=\int_K\dv\uw{}_h\dv\uv{}_h+\int_K\rot\uw{}_h\rot\uv{}_h;
\end{multline*}
here we have used the fact that $\partial_{\mathbf{n}}\uw{}_h$, $\dv\uw{}_h$ and $\rot\uw{}_h$ are all linear polynomials along the edges of $K$ and that the first order moments of $\uv{}_h$ vanish along the edges of $K$.

Now, given $\uw{}_h,\uv{}_h\in\uG{}^2_{ht}$, there exist uniquely $\uw{}_h^1,\uv{}_h^1\in\uS{}^2_{ht}$ and $\uw{}_h^2,\uv{}_h^2\in\undertilde{\mathcal{B}}{}^2_{h0}$, such that 
$$
\uw{}_h=\uw{}^1_h+\uw{}^2_h,\ \ \mbox{and}\ \ \uv{}_h=\uv{}^1_h+\uv{}^2_h.
$$
Thus 
$$
(\nabla_h\uw{}_h,\nabla_h\uv{}_h)=(\nabla_h\uw{}^1_h,\nabla_h\uv{}^1_h)+(\nabla_h\uw{}^1_h,\nabla_h\uv{}^2_h)+(\nabla_h\uw{}^2_h,\nabla_h\uv{}^1_h)+(\nabla_h\uw{}^2_h,\nabla_h\uv{}^2_h),
$$
and $(\dv_h\uw{}_h,\dv_h\uv{}_h)+(\rot_h\uw{}_h,\rot_h\uv{}_h)$ can be decomposed to four corresponding parts. Then \eqref{eq:h1=d+c} can be established for every pair of the parts, and the proof is completed. 
\end{proof}

\begin{remark}
It is known that \eqref{eq:h1=d+c} holds for $\uH{}^1_t$ functions but in general not for nonconforming finite element functions (such as the Crouzeix-Raviart element functions). This lemma reveals that the nonconforming space $\uG{}^2_{ht}$ is in some sense like a conforming one. 
\end{remark}

\subsection{An auxiliary finite element Stokes complex}

Given a grid $\mathcal{T}_h$, define
\begin{itemize}
\item $A^3_h:=\{w_h\in L^2(\Omega):w_h|_T\in P_3(T); w_h(a)\ \mbox{is\ continuous\ at}\ a\in\mathcal{X}_h\}$;
\item $A^3_{h0}:=\{w_h\in A^3_h:w_h(a)=0\ at\ a\in\mathcal{X}_h^b\};$
\item $\uG{}_h^{\rm 2,r}:=\{\uv\in(\mathbb{P}^2_h)^2;\ \int_e\llt \uv\cdot\mathbf{t}_e\rrt=0,\ \forall\,e\in \mathcal{E}_h^i\}$;
\item $\uG{}_{h0}^{\rm 2,r}:=\{\uv\in\uG{}_h^{\rm 2,r},\ \int_e \uv\cdot\mathbf{t}_e=0,\ \forall\,e\in \mathcal{E}_h^b\}.$
\end{itemize}

\begin{lemma}\label{lem:cmc}% cubic Morley complex
A finite element complex is given by
\begin{equation}
\begin{array}{ccccccccc}
0 & \xrightarrow{\bs{inclusion}} & A^3_{h0} & \xrightarrow{\bs{\mrm{\nabla}}_h} & \uG{}_{h0}^{\rm 2,r} & \xrightarrow{\mrm{rot}_h} & \mathbb{P}_{h0}^1  & \xrightarrow{\int\cdot} & 0.
\end{array}
\end{equation}
\end{lemma}
\begin{proof}
We adopt the standard counting technique.

Firstly, by Lemma \ref{lem:pGL}, $\mathbb{P}^1_{h0}=\rot_h\uG{}^2_{h0}\subset \rot_h\uG{}^{\rm 2,r}_{h0}\subset \mathbb{P}^1_{h0}$. Secondly, $\nabla_hA^3_{h0}\subset\{\uv{}_h\in\uG{}^{\rm 2,r}_{h0}:\rot_h\uv{}_h=0\}$. Thus we only have to check if $\dim(\nabla_hA^3_{h0})+\dim(\mathbb{P}^1_{h0})=\dim(\uG{}^{\rm 2,r}_{h0})$, which can be verified by observing that $\dim(A^3_{h0})=\#(\mathcal{X}_h^i)+7\#(\mathcal{T}_h)$, $\dim(\uG{}^{\rm 2,r}_{h0})=\#(\mathcal{E}_h^i)+9\#(\mathcal{T}_h)$ and $\dim(\uP{}^1_{h0})=3\#(\mathcal{T}_h)-1$, and by the Euler formula. The proof is completed. 
\end{proof}

\subsection{Finite element spaces for $H^2$ and discretized Stokes complexes}
Define
\begin{multline*}
B_h^{3}:=\{v\in  \mathbb{P}^3_h:\ v\ \mbox{is\ continuous\ at}\ a\in\mathcal{X}_h;
\int_e\llt v\rrt=0,\ \mbox{and}\  \int_ep_e\llt\partial_{\bf n}v\rrt=0, \forall\, p_e\in P_1(e),\ \forall\, e\in\mathcal{E}_h^i\},
\end{multline*}
%\begin{multline*}
$$
B_{ht}^{3}:=\{v\in B_h^{3}: v(a)=0,\ a\in\mathcal{X}_h^b;\ \int_ev=0,\   \forall\,e\in\mathcal{E}_h^b\},
$$
and
$$
B_{h0}^{3}:=\{v\in B_{ht}^{3}:\ \int_ep_e\partial_{\bf n}v=0,\  \forall\, p_e\in P_1(e),\ \forall\,e\in\mathcal{E}_h^b\}.
$$
According to the boundary conditions on $B^3_{h0}$, we can recognize them as for $H^2$ problems. 
\begin{remark}
Note that, given $v_h\in B^3_h$, on every cell, $v_h$ is embedded in 12 restrictions. We can not expect $B^3_h$ correspond to a finite element defined with Ciarlet's triple. 
\end{remark}

\begin{lemma}\label{eq:imp3}
$B^3_{ht}=\{w_h\in A^3_{h0}:\nabla_h w_h\in \uG{}^2_{ht}\}$, and $B^3_{h0}=\{w_h\in A^3_{h0}:\nabla_h w_h\in \uG{}^2_{h0}\}$. 
\end{lemma}
\begin{proof}
Firstly, by an elementary calculus, the continuity restriction of $B^3_h$ implies that $\int_ep_e\llbracket\partial_\mathbf{t}v_h\rrbracket=0$ for any $p_e\in P_1(e)$, any $e\in\mathcal{E}_h^i$ and any $v_h\in B^3_h$. Also, $\int_ep_e\partial_\mathbf{t}v_h=0$ for any $p_3\in P_1(e)$, any $e\in\mathcal{E}_h^b$ and any $v_h\in B^3_{ht}$.

By the definitions of $B^3_{h0}$ and $A^3_{h0}$, $B^3_{h0}\subset\{w_h\in A^3_{h0}:\nabla w_h\in \uG{}^2_{h0}\}$. On the other hand, given $w_h\in A^3_{h0}$ such that $\nabla_hw_h\in \uG{}^2_{h0}$, then $\int_e\llt\partial_{\bf n_e} w_h\rrt p_e=\int_e\llt\partial_{\bf t_e} w_h\rrt p_e=0$ for any $e\in\mathcal{E}_h$ and $p_e\in P_1(e)$. This implies $w_h\in B^3_{h0}$. Namely $B^3_{h0}=\{w_h\in A^3_{h0}:\nabla_h w_h\in \uG{}^2_{h0}\}$. Similarly can $B^3_{ht}=\{w_h\in A^3_{h0}:\nabla_h w_h\in \uG{}^2_{ht}\}$ be proved, and all the proof is completed. 
\end{proof}

\begin{lemma}\label{lem:hes=lap}
It holds for $w_h,v_h\in B^3_{ht}$ that 
\begin{equation}
(\nabla_h^2w_h,\nabla_h^2v_h)=(\Delta_hw_h,\Delta_hv_h). 
\end{equation}
\end{lemma}
\begin{proof}
By Lemma \ref{lem:h1=d+c}, as $\nabla_hB^3_{ht}\subset\uG{}^2_{ht}$,
$$
(\nabla_h^2w_h,\nabla_h^2v_h)=(\dv_h\nabla_hw_h,\dv_h\nabla_hv_h)+(\rot_h\nabla_hw_h,\rot_h\nabla_hv_h)=(\Delta_h w_h,\Delta_h v_h), \ \forall\,w_h,v_h\in B^3_{ht}.
$$
The proof is completed. 
\end{proof}
\begin{remark}
The lemma reveals that the functions in $B^3_{ht}$ possess some property like the $H^2$ conforming functions. 
\end{remark}

\begin{theorem}\label{thm:exactb3}
Two discretized Stokes complexex are given by
\begin{equation}\label{eq:complexbh3}
\begin{array}{ccccccccc}
0 & \xrightarrow{\bs{inclusion}} & B_{h0}^{3} & \xrightarrow{\bs{\mrm{\nabla}}_h} & \uG{}_{h0}^{2} & \xrightarrow{\mrm{rot}_h} & \mathbb{P}_{h0}^1  & \xrightarrow{\int\cdot} & 0.
\end{array}
\end{equation}
and
\begin{equation}\label{eq:complexbh3navier}
\begin{array}{ccccccccc}
0 & \xrightarrow{\bs{inclusion}} & B_{ht}^{3} & \xrightarrow{\bs{\mrm{\nabla}}_h} & \uG{}_{ht}^{2} & \xrightarrow{\mrm{rot}_h} & \mathbb{P}_{h0}^1  & \xrightarrow{\int\cdot} & 0.
\end{array}
\end{equation}

\end{theorem}
\begin{proof}
By Lemmas \ref{lem:pGL}, given $p_h\in\mathbb{P}^1_{h0}$, there exists $\usigma{}_h\in\uG{}^2_{h0}$, such that $\rot_h\usigma{}_h=p_h$. Further, given $\utau{}_h\in\uG{}^2_{h0}$, such that $\rot_h\utau{}_h=0$, by Lemma \ref{lem:cmc}, there exists $w_h\in A^3_{h0}$, such that $\utau{}_h=\nabla_h w_h$. Further, by Lemma \ref{eq:imp3}, $w_h\in B^3_{h0}$. Therefore, \eqref{eq:complexbh3} is proved. Similarly can \eqref{eq:complexbh3navier} be proved.  
\end{proof}
\begin{remark}
A key feature for the proof of Theorem \ref{thm:exactb3} is to construct a bigger finite element complex to cover, e.g., \eqref{eq:complexbh3}; this is accomplished by Lemma \ref{lem:cmc}, where a finite element complex is constructed where the same piecewise polynomial space with lower regularity is used corresponding to \eqref{eq:complexbh3}. A dual way can be to use bigger piecewise polynomial space with the same regularity. A different proof of \eqref{eq:complexbh3} can be found along this line in \cite{Zhang.S2018}.
\end{remark}

\section{Optimal nonconforming finite element schemes for biharmonic equation}
\label{sec:cubic}

We consider the biharmonic equation with $f\in L^2(\Omega)$:
\begin{equation}\label{eq:modeld}
\mbox{Dirichlet\ type}\quad
\left\{
\begin{array}{rl}
\Delta^2u=f&\mbox{in}\,\Omega,
\\
u=\partial_{\mathbf n}u=0,&\mbox{on}\,\partial\Omega,
\end{array}
\right.
\end{equation}
and
\begin{equation}\label{eq:modeln}
\mbox{Navier\ type}\quad
\left\{
\begin{array}{rl}
\Delta^2z=f&\mbox{in}\,\Omega;
\\
z=\Delta z=0,&\mbox{on}\,\partial\Omega.
\end{array}
\right.
\end{equation}
The variational problems are respectively 
\begin{itemize}
\item to find $u\in H^2_0(\Omega)$ such that 
\begin{equation}\label{eq:bhvp}
(\nabla^2u,\nabla^2v)=(f,v),\quad\forall\,v\in H^2_0(\Omega),
\end{equation}
\item to find $z\in H^2(\Omega)\cap H^1_0(\Omega)$, such that 
\begin{equation}\label{eq:bhvpn}
(\nabla^2z,\nabla^2v)=(f,v),\quad\forall\,v\in H^2\cap H^1_0(\Omega).
\end{equation}
\end{itemize}
In this section, we consider the nonconforming finite element discretization for them:
\begin{itemize}
\item
find $u_h\in B^3_{h0}$ such that 
\begin{equation}\label{eq:bhvpp3}
a_h(u_h,v_h):=(\nabla_h^2u_h,\nabla^2_hv_h)=(f,v_h),\quad\forall\,v_h\in B^3_{h0};
\end{equation}
\item find $z_h\in B^3_{ht}$ such that 
\begin{equation}\label{eq:bhvpp3n}
a_h(z_h,v_h)=(f,v_h),\quad\forall\,v_h\in B^3_{ht}.
\end{equation}
\end{itemize}
By the weak continuity of $B^3_{ht}$, $|\cdot|_{2,h}$ (namely, $\|\nabla_h^2\cdot\|_{0,\Omega}$) is a norm on $B^3_{ht}$, and \eqref{eq:bhvpp3} and \eqref{eq:bhvpp3n} are well-posed.

The main result of this section is contained in the theorem below.
\begin{theorem}\label{thm:ratep3}
Let $u$, $u_h$, $z$ and $z_h$ be solutions of \eqref{eq:bhvp} and \eqref{eq:bhvpp3}, \eqref{eq:bhvpn}, and \eqref{eq:bhvpp3n}, respectively. Then, with a generic constant $C$ depending on $\Omega$ and the regularity of the grid only, it holds for $u,z\in H^m(\Omega)$, $m=3,4$, that
\begin{equation}\label{eq:h2error3}
\|\nabla_h^2(u-u_h)\|_{0,\Omega}\leqslant C(h^{m-2}|u|_{m,\Omega}+h^2\|f\|_{0,\Omega}).
\end{equation}
and
\begin{equation}\label{eq:h2error3n}
\|\nabla_h^2(z-z_h)\|_{0,\Omega}\leqslant C(h^{m-2}|z|_{m,\Omega}+h^2\|f\|_{0,\Omega}).
\end{equation}
Moreover, when $\Omega$ is convex, 
\begin{equation}\label{eq:h1error3}
\|\nabla_h(u-u_h)\|_{0,\Omega}\leqslant C(h^{m-1}|u|_{m,\Omega}+h^3\|f\|_{0,\Omega}),
\end{equation}
and
\begin{equation}\label{eq:h1error3n}
\|\nabla_h(u-u_h)\|_{0,\Omega}\leqslant C(h^{m-2+\delta}|u|_{m,\Omega}+h^3\|f\|_{0,\Omega}),\ \ 1/2<\delta\leqslant 1.
\end{equation}
When $\Omega$ is specifically a rectangle, $\delta =1$.
\end{theorem}

We postpone the proof of Theorem \ref{thm:ratep3} after some technical lemmas. 

\subsection{Approximation property of $B^3_{h0}$}

First of all, we define an interpolator to $B^3_{h0}$. Given $w\in H^3(\Omega)\cap H^2_0(\Omega)$, set $\uphi:=\nabla w$, then $\uphi\in \uH^2(\Omega)\cap \uH{}^1_0(\Omega)$ and $\rot\,\uphi=0$. Indeed, $(\uphi,p\equiv 0)$ solves the incompressible Stokes equation:
\begin{equation}
\left\{
\begin{array}{lll}
(\nabla\uphi,\nabla\upsi)+(\rot\upsi,p)&=(-\Delta \uphi,\upsi),&\forall\,\upsi\in\uH{}^1_0(\Omega);
\\
(\rot\uphi,q)&=0,&\forall\,q\in L^2_0(\Omega).
\end{array}
\right.
\end{equation}
Now, choose $(\uphi{}_h,p_h)\in \uG{}^2_{h0}\times \mathbb{P}_{h0}^1$ such that 
\begin{equation}\label{eq:auxstokes}
\left\{
\begin{array}{lll}
(\nabla_h\uphi{}_h,\nabla_h\upsi{}_h)+(\rot_h\upsi{}_h,p_h)&=(-\Delta \uphi,\upsi{}_h),&\forall\,\upsi{}_h\in\uG{}^2_{h0};
\\
(\rot_h\uphi{}_h,q_h)&=0, &\forall\,q_h\in \mathbb{P}_{h0}^1.
\end{array}
\right.
\end{equation}
Then, by Theorem \ref{thm:exactb3}, there exists a unique $w_h\in B^3_{h0}$ such that $\nabla_hw_h=\uphi{}_h$. This way, we define an interpolation operator $\mathbb{I}_{h0}^B:H^3(\Omega)\cap H^2_0(\Omega)\to B^3_{h0}$ by
\begin{equation}
\mathbb{I}_{h0}^Bw:=w_h.
\end{equation}

%\subsubsection{Error estimation of $\mathbb{I}^B_{h0}$}

\begin{lemma}\label{lem:approxb3}
There exists a constant $C$ such that for any $w\in H^2_0(\Omega)\cap H^m(\Omega)$, $m=3,4$, it holds for $k=2$ that
\begin{equation}\label{eq:interrTnorm}
|w-\mathbb{I}^B_{h0}w|_{k,\Omega}^2\leqslant C\sum_{T\in\mathcal{T}_h}h_T^{2m-2k}|w|_{m,T}^2.
\end{equation}
If $\Omega$ is convex, then \eqref{eq:interrTnorm} holds for $k=1,2$.
\end{lemma}
\begin{proof}
By definition, the interpolation error of $\mathbb{I}^B_{h0}$ is the discretization error of \eqref{eq:auxstokes}, and \eqref{eq:interrTnorm} can be obtained by standard technique (with $\Omega$ either convex or nonconvex). 
\end{proof}

\subsection{Approximation of $B^3_{ht}$}

Again, we firstly define an interpolator to $B^3_{ht}$. Given $w\in H^3(\Omega)\cap H^1_0(\Omega)$ such that $\Delta w|_\Gamma=0$, set $\uphi:=\nabla w$, then $\uphi\in \uH^2(\Omega)\cap H_0(\rot,\Omega)$, $\rot\,\uphi=0$ and $(\dv\uphi)|_\Gamma=0$. Indeed, $(\uphi,p\equiv 0)$ solves the incompressible Stokes equation:
\begin{equation}\label{eq:stokest}
\left\{
\begin{array}{lll}
(\nabla\uphi,\nabla\upsi)+(\rot\upsi,p)&=(-\Delta \uphi,\upsi),&\forall\,\upsi\in\uH{}^1(\Omega)\cap H_0(\rot,\Omega);
\\
(\rot\uphi,q)&=0,&\forall\,q\in L^2_0(\Omega).
\end{array}
\right.
\end{equation}
Now, choose $(\uphi{}_h,p_h)\in \uG{}^2_{ht}\times \mathbb{P}_{h0}^1$ such that 
\begin{equation}\label{eq:auxstokesn}
\left\{
\begin{array}{lll}
(\nabla_h\uphi{}_h,\nabla_h\upsi{}_h)+(\rot_h\upsi{}_h,p_h)&=(-\Delta \uphi,\upsi{}_h),&\forall\,\upsi{}_h\in\uG{}^2_{ht};
\\
(\rot_h\uphi{}_h,q_h)&=0, &\forall\,q_h\in \mathbb{P}_{h0}^1.
\end{array}
\right.
\end{equation}
Then, by Theorem \ref{thm:exactb3}, there exists a unique $w_h\in B^3_{ht}$ such that $\nabla_hw_h=\uphi{}_h$. This way, we define an interpolation operator $\mathbb{I}_{ht}^B:H^3(\Omega)\cap H^1_0(\Omega)\to B^3_{ht}$ by
\begin{equation}
\mathbb{I}_{ht}^Bw:=w_h.
\end{equation}

%\subsubsection{Error estimation of $\mathbb{I}^B_{h0}$}

\begin{lemma}\label{lem:approxb3n}
There exists a constant $C$ such that for any $w\in H^1_0(\Omega)\cap H^m(\Omega)$ such that $\Delta w|_\Gamma=0$, $m=3,4$, it holds that
\begin{equation}\label{eq:interrTnormn}
|w-\mathbb{I}^B_{ht}w|_{2,\Omega}^2\leqslant C\sum_{T\in\mathcal{T}_h}h_T^{2m-4}|w|_{m,T}^2.
\end{equation}
If $\Omega$ is convex, then
\begin{equation}\label{eq:interrTnormnh1}
|w-\mathbb{I}^B_{ht}w|_{1,\Omega}^2\leqslant C\sum_{T\in\mathcal{T}_h}h_T^{2m-4+\kappa}|w|_{m,T}^2,\ \mbox{with}\ 1<\kappa\leqslant2.
\end{equation}
If specifically $\Omega$ is rectangle, $\kappa=2$.
\end{lemma}
\begin{proof}
By definition, the interpolation error of $\mathbb{I}^B_{ht}$ is the discretization error of \eqref{eq:auxstokesn}, and \eqref{eq:interrTnormn} and \eqref{eq:interrTnormnh1} can be obtained by standard technique (with $\Omega$ either convex or nonconvex). We only have to note that the regularity of the auxiliary Stokes problem on convexs domain can be affected under the boundary condition of this kind. Specifically, we refer to \cite{Blum.H;Rannacher.R1980} for the full regularity of \eqref{eq:modeln} and thus of the auxiliary Stokes problem \eqref{eq:stokest} on rectangles. 
\end{proof}

\subsection{Convergence analysis of the nonconforming scheme}

For suitable $\varphi$ and $\psi$, define the bilinear forms
\begin{equation}\label{eq:res1def}
\mathcal{R}_h^1(\varphi,\psi):=(\nabla^2\varphi,\nabla_h^2\psi)+(\nabla\Delta \varphi,\nabla_h\psi),
\end{equation}
\begin{equation}\label{eq:res2def}
\mathcal{R}_h^2(\varphi,\psi):=(\nabla\Delta \varphi,\nabla_h\psi)+(\Delta^2\varphi,\psi),
\end{equation}
and 
\begin{equation}\label{eq:resdef}
\mathcal{R}_h(\varphi,\psi):=\mathcal{R}_h^1(\varphi,\psi)-\mathcal{R}_h^2(\varphi,\psi).
\end{equation}

\begin{lemma}\label{lem:res}
There exists a constant $C$ such that it holds for any $\varphi\in H^2_0(\Omega)\cap H^k(\Omega)$, $w_h\in B^3_{h0}+H^2_0(\Omega)$, and $k=3,4$ that,
\begin{equation}\label{eq:res1}
\displaystyle\mathcal{R}_h^1(\varphi,w_h)\leqslant Ch^{k-2}|\varphi|_{k,\Omega}\|\nabla_h^2w_h\|_{0,\Omega},
\end{equation}

%\begin{multline}
\begin{equation}
\label{eq:res2}
\displaystyle\mathcal{R}_h^2(\varphi,w_h)\leqslant Ch^{k-2}(|\varphi|_{k,\Omega}+h^2\|\Delta^2\varphi\|_{0,\Omega})\|\nabla_h^2w_h\|_{0,\Omega},
\end{equation}
%\end{multline}
\end{lemma}

\begin{proof}
Given $e\in\mathcal{E}_h$, by the definition of $B_{h0}^3$, $\fint_{e}p_e\llbracket\partial_{\mathbf{n}_e}w_h\rrbracket_e=0,\ p_e\in P_1(e)$; for the tangential direction, $\fint_{e}p_e\llbracket\partial_{\mathbf t_e}w_h\rrbracket_e=(p_e(L_e)\llbracket w_h\rrbracket_e(L_e)-p_e(R_e)\llbracket w_h\rrbracket_e(R_e))-\fint_{e}\partial_{\mathbf t_e}p_e\llbracket w_h\rrbracket_e=0$. Hence, 
\begin{equation}\label{eq:momentc3+}
\fint_{e}p_e\llbracket\nabla w_h\rrbracket_e=\undertilde{0},\ \forall\, p_e\in P_1(e),\ \ e\in\mathcal{E}_h.
\end{equation}
Therefore, \eqref{eq:res1} follows by standard techniques.

Now, define $\Pi_h^2$ the nodal interpolation to $\mathcal{L}_{h0}^2$ by
$$
(\Pi_h^2w)(a)=w(a),\ \forall\,a\in\mathcal{X}_h^i;\quad \fint_e(\Pi_h^2w)=\fint_ew,\ \forall\,e\in\mathcal{E}_h^i.
$$
It is easy to verify that the operator is well-defined. Moreover,
\begin{equation}\label{eq:orthpi2}
\fint_T\undertilde{c}\cdot\nabla (w-\Pi_h^2w)=0,\ \ \forall\,\undertilde{c}\in\mathbb{R}^2\ \mbox{and}\ T\in\mathcal{T}_h, \ \mbox{provided}\ w\in H^2_0(\Omega)+B^3_{h0}.
\end{equation}
By Green's formula,
\begin{equation}
(\Delta^2u,\Pi_h^2 w_h)=-(\nabla\Delta u,\nabla \Pi_h^2 w_h).
\end{equation}
Therefore, 
$$
\displaystyle\mathcal{R}_h^2(\varphi,w_h)=(\nabla \Delta u,\nabla_h(w_h-\Pi_h^2w_h))+(\Delta^2u,w_h-\Pi_h^2w_h):=I_1+I_2.
$$
By \eqref{eq:orthpi2},
$$
I_1=\inf_{\undertilde{c}\in(\mathbb{P}_h^0)^2}\left(\left[\nabla\Delta u-\undertilde{c}\right],\nabla_h(\Pi_h^2w_h-w_h)\right)\leqslant C(h^{k-2}|u|_{k,\Omega}+h^2\|\Delta u\|_{0,\Omega})\|\nabla_h^2w_h\|_{0,\Omega}.
$$
Further, 
$$
I_2\leqslant Ch^2\|\Delta^2u\|_{0,\Omega}\|\nabla_h^2w_h\|_{0,\Omega}.
$$
Summing all above proves \eqref{eq:res2}. 
\end{proof}

Similarly, we have the lemma below. 
\begin{lemma}\label{lem:resn}
There exists a constant $C$ such that it holds for any $\varphi\in H^1_0(\Omega)\cap H^k(\Omega)$ so that $(\Delta \varphi)|_\Gamma=0$, $w_h\in B^3_{ht}+H^2(\Omega)\cap H^1_0(\Omega)$ and $k=3,4$ that, 
\begin{equation}\label{eq:res1n}
\displaystyle\mathcal{R}_h^1(\varphi,w_h)\leqslant Ch^{k-2}|\varphi|_{k,\Omega}\|\nabla_h^2w_h\|_{0,\Omega},
\end{equation}

%\begin{multline}
\begin{equation}
\label{eq:res2}
\displaystyle\mathcal{R}_h^2(\varphi,w_h)\leqslant Ch^{k-2}(|\varphi|_{k,\Omega}+h^2\|\Delta^2\varphi\|_{0,\Omega})\|\nabla_h^2w_h\|_{0,\Omega}.
\end{equation}
%\end{multline}
\end{lemma}

\paragraph{\bf Proof of Theorem \ref{thm:ratep3}} 
The proof follows a similar approach as the one in \cite{Shi.Z1990}, with some technical modifications. By Strang lemma,
$$
\|\nabla_h^2(u-u_h)\|_{0,\Omega}\cequiv \inf_{v_h\in B^3_{h0}}\|\nabla_h^2(u-v_h)\|_{0,\Omega}+\sup_{v_h\in B^3_{h0}\setminus\{\mathbf{0}\}}\frac{(\nabla^2u,\nabla^2_hv_h)-(f,v_h)}{\|\nabla_h^2v_h\|_{0,\Omega}}.
$$
The approximation error estimate follows by Lemma \ref{lem:approxb3}. By Lemma \ref{lem:res}, 
$$
(\nabla^2u,\nabla^2_hv_h)-(f,v_h)=(\nabla^2u,\nabla^2_hv_h)-(\Delta^2u,v_h)=\mathcal{R}_h(u,v_h)\leqslant Ch^2|u|_{4,\Omega}\|\nabla_h^2v_h\|_{0,\Omega},
$$which completes the proof of \eqref{eq:h2error3}.

Now, we turn our attention to the proof of \eqref{eq:h1error3} for convex $\Omega$. Denote $u^\Pi_h=\mathbb{I}^B_{h0}u$. Then, by Lemma \ref{lem:approxb3}, $\|\nabla_h^j(u-u^\Pi_h)\|_{0,\Omega}\leqslant Ch^{4-j}|u|_{4,\Omega}$, $j=1,2$. Denote by $\Pi_h^1$ the nodal interpolation onto $\mathcal{L}^1_{h0}$, then $\Pi_h^1(u^\Pi_h-u_h)\in H^1_0(\Omega)$. Set $\varphi\in H^3(\Omega)\cap H^2_0(\Omega)$ such that 
$$
(\nabla^2\varphi,\nabla^2 v)=(\nabla\Pi_h^1(u^\Pi_h-u_h),\nabla v),\quad\forall\,v\in H^2_0(\Omega),
$$
then when $\Omega$ is convex, $\|\varphi\|_{3,\Omega}\cequiv \|\Pi_h^1(u^\Pi_h-u_h)\|_{1,\Omega}$. By Green's formula, 
\begin{multline*}\label{eq:shi23'}
\|\nabla\Pi_h^1(u^\Pi_h-u_h)\|_{0,\Omega}^2=-(\nabla\Delta\varphi,\nabla \Pi_h^1(u_h^\Pi-u_h))=-(\nabla\Delta\varphi,\nabla\Pi_h^1(u_h^\Pi-u))-(\nabla\Delta \varphi,\nabla\Pi_h^1(u-u_h))
\\
=(\nabla\Delta\varphi\cdot\nabla ({\rm Id}-\Pi_h^1)(u^\Pi_h-u_h))-(\nabla\Delta \varphi\cdot\nabla(u^\Pi_h-u))-(\nabla\Delta \varphi\cdot\nabla(u-u_h)):=I_1+I_2+I_3.
\end{multline*}
Further, set $\varphi_h^\Pi=\mathbb{I}^B_{h0}\varphi$, and
\begin{multline*}
I_3=(\nabla^2\varphi,\nabla_h^2(u-u_h))+\mathcal{R}_h^1(\varphi,u-u_h)
=
-(\nabla_h^2(\varphi-\varphi_h^\Pi),\nabla_h^2(u-u_h))-(\nabla_h^2\varphi_h^\Pi,\nabla_h^2(u-u_h))+\mathcal{R}_h^1(\varphi,u-u_h)
\\
=
-(\nabla_h^2(\varphi-\varphi_h^\Pi),\nabla_h^2(u-u_h))+\mathcal{R}_h(u,\varphi-\varphi_h^\Pi)+\mathcal{R}_h^1(\varphi,u-u_h).
\end{multline*}
Therefore, $\|\nabla\Pi_h^1(u^\Pi_h-u_h)\|_{0,\Omega}^2\leqslant C|\varphi|_{3,\Omega}(h^{m-1}|u|_{m,\Omega}+h^3\|\Delta^2 u\|_{0,\Omega}),$ and $\|\nabla\Pi_h^1(u_h^\Pi-u_h)\|_{0,\Omega}\leqslant C(h^{m-1}|u|_{m,\Omega}+h^3\|\Delta^2 u\|_{0,\Omega})$. Finally,
\begin{multline*}
\|\nabla_h(u-u_h)\|_{0,\Omega}\leqslant \|\nabla_h(u-u_h^\Pi)\|_{0,\Omega}+\|\nabla_h(u_h^\Pi-u_h)\|_{0,\Omega}
\\
\leqslant  
\|\nabla_h(u-u_h^\Pi)\|_{0,\Omega}+\|\nabla_h[(u_h^\Pi-u_h)-\Pi_h^1(u_h^\Pi-u_h)]\|_{0,\Omega}+\|\nabla\Pi_h^1(u_h^\Pi-u_h)\|_{0,\Omega} 
\\
\leqslant C(h^{m-1}|u|_{m,\Omega}+h^3\|\Delta^2u\|_{0,\Omega}).\qquad
\end{multline*}

The proof of \eqref{eq:h1error3n} and \eqref{eq:h2error3n} is basically the same. The convergence rate for the $H^1$ norm of the error is slightly lost due to the lost of the regularity of the model problem \eqref{eq:modeld} on general convex polygons. The proof is completed.\qed

\subsection{A variant formulation for bi-Laplacian equation with varying coefficient}

The bi-Laplacian equation $\Delta (\mathcal{A}\Delta u)=f$, where $\mathcal{A}$ is a non-constant coefficient with positive lower and upper bounds, is frequently dealt with in applications. The equation arises in, e.g., the Helmholtz transmission eigenvalue problem in acoustics (c.f., e.g., \cite{Colton.D;Monk.P1988,Kirsch.A1986,Xi.Y;Ji.X;Zhang.S2019+}). The variational problem is to find $u\in H^2_0(\Omega)$ such that 
\begin{equation}\label{eq:bhvpnd}
(\mathcal{A}\Delta u,\Delta v)=(f,v),\quad\forall\,v\in H^2_0(\Omega).
\end{equation}
Correspondingly, we consider the nonconforming finite element discretization:
\begin{quote} 
find $u_h\in B^3_{h0}$ such that 
\begin{equation}\label{eq:bhvpp3nd}
\tilde{a}_h(u_h,v_h):=(\mathcal{A}\Delta_hu_h,\Delta_hv_h)=(f,v_h),\quad\forall\,v_h\in B^3_{h0}.
\end{equation}
\end{quote}
\begin{lemma}\label{lem:wpnd}
The finite element problem \eqref{eq:bhvpp3nd} admits a unique solution.
\end{lemma}
\begin{proof}
By Lemma \ref{lem:hes=lap},  the bilinear form $\tilde{a}_h(\cdot,\cdot)$ is coercive on $B^3_{h0}$ with respect to the norm $|\cdot|_{2,h}$. The well-posedness of  \eqref{eq:bhvpp3nd} follows by Lax-Milgrem lemma. The proof is completed. 
\end{proof}

Similar to Theorem \ref{thm:ratep3}, we can establish and prove the theorem below. 
\begin{theorem}\label{thm:ratep3nd}
Let $u$ and $u_h$ be solutions of \eqref{eq:bhvpnd} and \eqref{eq:bhvpp3nd}, respectively. Then, with a generic constant $C$ depending on $\mathcal{A}$, $\Omega$ and the regularity of the grid only, it holds for $u\in H^m(\Omega)$, $m=3,4$, that
\begin{equation}\label{eq:h2error3nd}
\|\nabla_h^2(u-u_h)\|_{0,\Omega}\leqslant C(h^{m-2}|u|_{m,\Omega}+h^2\|f\|_{0,\Omega}).
\end{equation}
Moreover, when $\Omega$ is convex, 
\begin{equation}\label{eq:h1error3nd}
\|\nabla_h(u-u_h)\|_{0,\Omega}\leqslant C(h^{m-1}|u|_{m,\Omega}+h^3\|f\|_{0,\Omega}).
\end{equation}
\end{theorem}

\begin{remark}
For the bi-Laplacian equation with non-constant coefficient $\mathcal{A}$, the finite element scheme of the formulation \eqref{eq:bhvpp3nd} is a natural alternative. When the formulation \eqref{eq:bhvpp3nd} is used on, e.g., the Morley element, however, the scheme is not well-posed without extra stabilisations. Higher regularity of $B^3_{h0}$ here makes it fit for the formulation \eqref{eq:bhvpp3nd}.
\end{remark}

\begin{remark}\label{rem:laplapnavier}
Similarly, by Lemma \ref{lem:hes=lap}, a bilinear form induced by $(\mathcal{A}\Delta_h\cdot,\Delta_h\cdot)$ can be used for $\Delta(\mathcal{A}\Delta u)=f$ with Navier type boundary condition. 
\end{remark}

\section{On the implementation of the schemes}
\label{sec:imp}

In this section, we present two approaches to implement the schemes. One is to figure out the locally supported basis functions of $B^3_{h0}$ and $B^3_{ht}$, and the other is to decompose the finite element system to three sub-problems to be solved sequentially. The former approach makes the scheme fit for the general finite element programing procedure, and the latter approach, as the sub-problems are Poisson systems and a Stokes system, makes the finite element problems optimally solvable. 
\subsection{Locally supported basis functions of the finite element spaces}

\subsubsection{Structure of weakly rot-free space}

Let $T$ be a triangle with $a_i$ and $e_i$, $i=1,2,3$, being its vertices and edges. Define $\uP{}_2(\rot,w0):=\{\up\in\uP{}_2(T):\int_T\rot\up=0\}$. Then $\dim(\uP{}_2(\rot,w0))=11$. 

Denote

\fbox{
	\begin{minipage}{0.97\textwidth}
\begin{description}
\item[$\ueta{}_{a_i}^x$] such that $\ueta{}_{a_i}^x(a_j)=(\delta_{ij},0)^\top$; $\fint_{e_k}\ueta{}_{a_i}^x=\undertilde{0}$; $i,j,k=1,2,3$;

\item[$\ueta{}_{a_i}^y$] such that $\ueta{}_{a_i}^y(a_j)=(0,\delta_{ij})^\top$; $\fint_{e_k}\ueta{}^y_{a_i}=\undertilde{0}$; $i,j,k=1,2,3$.

\item[$\ueta{}_{a_i}$] such that $\ueta{}_{a_i}(a_j)=\undertilde{0}$;  $\fint_{e_k}\ueta{}_{a_i}\cdot\mathbf{t}_{e_k,a_i}=1-\delta_{ik}$ and $\int_{e_k}\ueta{}_{a_i}\cdot\mathbf{n}_{e_k,a_i}=0$, where $\mathbf{t}_{e_k,a_i}$ is the unit tangential vector along $e_k$ starting from $a_i$ and $\mathbf{n}_{e_k,a_i}$ is the normal direction of $e_k$; $i,j,k=1,2,3$;

\item[$\ueta{}_{e_i}$] such that $\ueta{}_{e_i}(a_j)=\undertilde{0}$; $\fint_{e_k}\ueta{}_{e_i}\cdot\utau{}_{e_k}=0$, $\fint_{e_k}\ueta{}_{e_i}\cdot\mathbf{n}_{e_k}=\delta_{ik}$; $i,j,k=1,2,3$.
\end{description}
\end{minipage}}
The functions form a frame of $\uP{}_2(\rot,w0)$. Indeed, $\ueta{}_{e_1}+\ueta{}_{e_2}+\ueta{}_{e_3}=0$, while we have the lemma below.
\begin{lemma}\label{lem:p2w0frame}
All $\ueta{}^x_{a_i}$, $\ueta{}^y_{a_i}$ and $\ueta{}_{e_i}$ for $i=1,2,3$ and any two of $\ueta{}_{a_i}$ among $i=1,2,3$ form a basis of $\uP{}_2(\rot,w0)$. 
\end{lemma}

Analogically, denote  $\uS{}^2_h(\rot,w0):=\{\uv\in \uS{}^2_h:(\rot\uv,q)=0,\ \forall\,q\in\mathbb{P}^0_h\}$, $\uS{}^2_{h0}(\rot,w0):=\{\uv\in\uS{}^2_{h0}:(\rot\uv,q)=0,\ \forall\,q\in\mathbb{P}^0_{h0}\}$ and $\uS{}^2_{ht}(\rot,w0):=\{\uv\in\uS{}^2_{ht}:(\rot\uv,q)=0,\ \forall\,q\in\mathbb{P}^0_{h0}\}$.

Meanwhile, for $a\in\mathcal{X}_h$, denote by $P_a$ the union of triangles of which $a$ is a vertex, namely the patch associated with $a$; for $e\in\mathcal{E}_h$,  denote by $P_e$ the patch associated with $e$. Denote, with respect to $a\in\mathcal{X}_h$ and $e\in\mathcal{E}_h$, functions in $\uS{}^2_h$ as,

\fbox{
	\begin{minipage}{0.97\textwidth}
\begin{description}
\item[$\uphi{}_a^x$] such that $\uphi{}_a^x(a)=(1,0)^\top$; $\uphi{}_a^x(a')=\undertilde{0}$ on $a\neq a'\in\mathcal{X}_h$; $\fint_{e'}\uphi{}_a^x=\undertilde{0}$ on $e'\in\mathcal{E}_h$;

\item[$\uphi{}_a^y$] such that $\uphi{}_a^y(a)=(0,1)^\top$; $\uphi{}_a^y(a')=\undertilde{0}$ on $a\neq a'\in\mathcal{X}_h$; $\fint_{e'}\uphi{}^y_a=\undertilde{0}$ on $e'\in\mathcal{E}_h^i$;

\item[$\uphi{}_{P_a}$] such that $\uphi{}_{P_a}(a')=\undertilde{0}$ on $ a'\in\mathcal{X}_h$; $\fint_e\uphi{}_{P_a}=\undertilde{0}$ on $e\in\mathcal{E}_h$ and $a\not\in e$; $\fint_e\uphi{}_{P_a}\cdot\mathbf{t}_{e,P_a}=1$ and $\int_e\uphi{}_{P_a}\cdot\mathbf{n}_{e,P_a}=0$ on $e\subset P_a$ and $a\in e$, where $\mathbf{t}_{e,P_a}$ is the unit tangential vector along $e$ starting from $a$ and $\mathbf{n}_{e,P_a}$ is the anticlockwise normal direction of $e$ with respect to $P_a$;

\item[$\uphi{}_e$] such that $\fint_e\uphi{}_e\cdot\utau{}_e=0$, $\fint_e\uphi{}_e\cdot\mathbf{n}_e=1$, and $\uphi{}_e$ vanishes on $\Omega\setminus \mathring{P_e}$.
\end{description}
\end{minipage}}

\begin{lemma}\label{eq:strucs2h0w0}
The set $\{\uphi{}^x_a,\uphi{}^y_a,\uphi{}_{P_a},\uphi{}_e \}_{a\in\mathcal{X}^i_h,\ e\in\mathcal{E}^i_h}$ forms a basis of $\uS{}^2_{h0}(\rot,w0)$; namely
\begin{equation}\label{eq:basisweakrotfree}
\uS{}^2_{h0}(\rot,w0)={\rm span}\{\uphi{}_a^x \}_{a\in\mathcal{X}_h^i}\oplus {\rm span}\{\uphi{}_a^y \}_{a\in\mathcal{X}_h^i}\oplus{\rm span}\{\uphi{}_e \}_{e\in\mathcal{E}_h^i}\oplus{\rm span}\{\uphi{}_{P_a} \}_{a\in\mathcal{X}_h^i}.
\end{equation}
\end{lemma}
\begin{proof}
By direct calculation, the functions $\uphi{}_a^x$, $\uphi{}_a^y$, $\uphi{}_{P_a}$ and $\uphi{}_e$ all belong to $\uS{}^2_{h0}(\rot,w0)$. By their definitions, the functions $\{\uphi{}^x_a,\uphi{}^y_a,\uphi{}_e \}_{a\in\mathcal{X}_h^i,\ e\in\mathcal{E}_h^i}$ are linearly independent, and the summation ${\rm span}\{\uphi{}_{P_a} \}_{a\in\mathcal{X}_h^i}+{\rm span}\{\uphi{}^x_a,\uphi{}^y_a,\uphi{}_e \}_{a\in\mathcal{X}_h^i,\ e\in\mathcal{E}_h^i}$ is direct. Since $\dim(\uS{}^2_{h0}(\rot,w0))=\dim(\uS{}^2_{h0})-\dim(\mathbb{P}^0_{h0})=\#(\mathcal(N)_{h0}^i)+3\#(\mathcal{X}_{h0}^i)$, it remains for us to show $\{\uphi{}_e \}_{e\in\mathcal{E}_h^i}$ are linearly independent. 

Assume there exist $\{\alpha_a \}_{a\in\mathcal{X}}\subset\mathbb{R}$ with $\alpha_a=0$ for $a\in\mathcal{X}_h^b$, such that  $\upsi=\sum_{a\in \mathcal{X}_{h}}\alpha_a\uphi{}_{P_a}\equiv 0$. By the definition of $\uphi{}_{P_a}$, for any $e\in\mathcal{E}_h^i$, $|\fint_{e}\upsi\cdot\mathbf{n}_e|=|\alpha_{a_e^L}-\alpha_{a_e^R}|$, where $a_e^L$ and $a_e^R$ are the two ends of $e$; thus $\alpha_{a_e^L}=\alpha_{a_e^R}$ for every $e\in\mathcal{E}_h^i$. Since $\alpha_a=0$ for $a\in \mathcal{X}_h^b$, $\alpha_a=0$ for $a\in\mathcal{X}_h^{b,+1}$; recursively, we obtain $\alpha_a=0$ for $a\in\mathcal{X}_h^{b,+j}$ level by level, and finally $\alpha_a=0$ for $a\in\mathcal{X}_h$. 

The proof is completed by noting the two sides of \eqref{eq:basisweakrotfree} have the same dimension. 
\end{proof}

For $a\in\mathcal{X}_h^b\setminus\mathcal{X}^c_h$, denote by $\mathbf{n}^b_a$ the outward unit normal vector of $\partial\Omega$ at $a$. Thus, for $a\in\mathcal{X}_h^b\setminus\mathcal{X}^c_h$, denote 
\begin{equation}
\uphi{}_a^b:=(\mathbf{n}^b_a)_x\uphi{}_a^x+(\mathbf{n}^b_a)_y\uphi{}_a^y.
\end{equation}
\begin{lemma}\label{eq:strucs2htw0}
$\uS{}^2_{ht}(\rot,w0)=\uS{}^2_{h0}(\rot,w0)\oplus{\rm span}\{\uphi{}_a^b:a\in \mathcal{X}_h^b\setminus\mathcal{X}^c_h\}\oplus{\rm span}\{\uphi{}_e:e\in\mathcal{E}_h^b\}$.
\end{lemma}

\begin{proof}
By Lemma \ref{lem:p2w0frame}, $\uphi{}^b_a$ with $a\in \mathcal{X}_h^b\setminus\mathcal{X}^c_h$ are linearly independent, and the right hand side is a direct sum included in the left hand side. On the other hand, 
\begin{multline*}
\dim(\uS{}^2_{ht}(\rot,w0))=\dim(\uS{}^2_{ht})-\dim(\mathbb{P}^0_{h0})
=\dim(\uS{}^2_{h0}(\rot,w0))+\dim(\mathcal{X}_h^b\setminus\mathcal{X}^c_h)+\dim(\mathcal{E}_h^b)
\\
=\dim(\uS{}^2_{h0}(\rot,w0))+\dim({\rm span}\{\uphi{}_a^b:a\in \mathcal{X}_h^b\setminus\mathcal{X}^c_h\})+\dim({\rm span}\{\uphi{}_e:e\in\mathcal{E}_h^b\})=\dim(\mbox{right\ hand\ side}).
\end{multline*}
This proves the assertion. 
\end{proof}

\subsubsection{Structure of piecewise rot-free space}

Denote $\uP{}_2(\rot,0)=\{\up\in\uP{}_2(T):\rot\up=0\}$,  and $\dim(\uP{}_2(\rot,0))=9$. Denote by $\phi_T$ the bubble function $(\lambda_1^2+\lambda_2^2+\lambda_3^2)-2/3$, and define a mapping $\mathcal{F}_T$ from $\uP{}_2(\rot,w0)$ to $\uP{}_2(\rot,0)$ by
$$
\mathcal{F}_T\ueta=\ueta+\undertilde{\phi},\ \undertilde{\phi}\in{\rm span}\{(\phi_T,0),(0,\phi_T)\},\ \ \ \mbox{such\ that}\ \ \rot (\mathcal{F}_T\ueta)=0.
$$
Since $\int_T\rot\ueta=0$ in the formula above, the mapping is well defined. It can be verified that $\mathcal{F}_T(\ueta{}^x_{a_1}+\ueta{}^x_{a_2}+\ueta{}^x_{a_3})=0$ and $\mathcal{F}_T(\ueta{}^y_{a_1}+\ueta{}^y_{a_2}+\ueta{}^y_{a_3})=0$. A frame of $\uP{}_2(\rot,0)$ is presented in the lemma below. 
\begin{lemma}\label{lem:p20frame}
 Any two $\mathcal{F}_T\ueta{}^x_{a_i}$ among $i=1,2,3$, any two $\mathcal{F}_T\ueta{}^y_{a_i}$ among $i=1,2,3$, any two $\mathcal{F}_T\ueta{}_{a_i}$ among $i=1,2,3$, and all $\mathcal{F}_T\ueta{}_{e_i}$ for $i=1,2,3$ form a basis of $\uP{}_2(\rot,0)$.
\end{lemma}
%\subsubsection{Structure of global rot-free spaces}
Denote $\uG{}^2_h(\rot,0):=\{\uv\in\uG{}^2_h:\rot_h\uv=0\}$, $\uG{}^2_{h0}(\rot,0):=\{\uv\in\uG{}^2_{h0}:\rot_h\uv=0\}$ and $\uG{}^2_{ht}(\rot,0):=\{\uv\in\uG{}^2_{ht}:\rot_h\uv=0\}$. Define an operator $\mathcal{F}_h:\uS{}^2_h(\rot,w0)\to \uG{}^2_h(\rot,0)$ by
\begin{equation}
\mathcal{F}_h\uphi{}_h=\uphi{}_h+\undertilde{\phi}{}_h,\ \undertilde{\phi}{}_h\in\undertilde{\mathcal{B}}{}^2_{h0},\ \ \ \mbox{such\ that}\ \ \rot_h (\mathcal{F}_h\uphi{}_h)=0.
\end{equation}
Since $\int_T\rot \uphi{}_h=0$ on any $T$ for $\uphi{}_h\in\uS{}^2_h(\rot,w0)$, $\mathcal{F}_h$ is well defined. Indeed,  $(\mathcal{F}_h\uphi{}_h)|_T=\mathcal{F}_T(\uphi{}_h|_T)$.

\begin{lemma}\label{eq:kernelbijec}
$\mathcal{F}_h$ is a bijection between $\uS{}^2_{h0}(\rot,w0):=\{\uv\in\uS{}^2_{h0}:(\rot\uv,q)=0,\ \forall\,q\in\mathbb{P}^0_{h0}\}$ and $\uG{}^2_{h0}(\rot,0):=\{\uv\in\uG{}^2_{h0}:\rot_h\uv=0\}$, and a bijection between $\uS{}^2_{ht}(\rot,w0):=\{\uv\in\uS{}^2_{ht}:(\rot\uv,q)=0,\ \forall\,q\in\mathbb{P}^0_{h0}\}$ and $\uG{}^2_{ht}(\rot,0):=\{\uv\in\uG{}^2_{ht}:\rot_h\uv=0\}$.
\end{lemma}
\begin{proof}
Since $\uS{}^2_{ht}\cap\undertilde{\mathcal{B}}{}_{h0}=\{0\}$, $\mathcal{F}_h$ is an injection on $\uS{}^2_{ht}(\rot,w0)$. 

Given $\ugamma{}_h\in\uG{}^2_{ht}(\rot,0)$, decompose it to $\ugamma{}_h=\ugamma{}_h^1+\ugamma{}_h^2$ such that $\ugamma{}_h^1\in\uS{}^2_{ht}$ and $\ugamma{}_h^2\in\undertilde{\mathcal{B}}{}_{h0}$. As $\rot(\ugamma{}_h|_T)=0$ and $\int_T\rot(\phi_T,0)=\int_T\rot(0,\phi_T)=0$ on every cell $T$, $\int_T\rot(\ugamma{}_h^1|_T)=0$. Namely $\ugamma{}_h^1\in \uS{}^2_{ht}(\rot,w0)$. This way $\mathcal{F}_h$ is a bijection between $\uS{}^2_{ht}(\rot,w0)$ and $\uG{}^2_{ht}(\rot,0)$.

Similarly we can prove $\mathcal{F}_h\uS{}^2_{h0}(\rot,w0)=\uG{}^2_{h0}(\rot,0)$, and the proof is completed. 
\end{proof}

By Lemmas \ref{eq:strucs2h0w0}, \ref{eq:strucs2h0w0} and \ref{eq:kernelbijec}, we can prove the lemmas below. 
\begin{lemma}\label{eq:strucG2h0w0}
The set $\{\mathcal{F}_h\uphi{}^x_a,\mathcal{F}_h\uphi{}^y_a,\mathcal{F}_h\uphi{}_{P_a},\mathcal{F}_h\uphi{}_e \}_{a\in\mathcal{X}^i_h,\ e\in\mathcal{E}^i_h}$ forms a basis of $\uG{}^2_{h0}(\rot,0)$.
\end{lemma}

\begin{lemma}\label{eq:strucG2htw0}
$\uG{}^2_{ht}(\rot,0)=\uG{}^2_{h0}(\rot,0)\oplus{\rm span}\{\mathcal{F}_h\uphi{}_a^b:a\in \mathcal{X}_h^b\setminus\mathcal{X}^c_h\}\oplus{\rm span}\{\mathcal{F}_h\uphi{}_e:e\in\mathcal{E}_h^b\}$.
\end{lemma}

\subsubsection{Locally supported basis functions of $B^3_{h0}$ and $B^3_{ht}$}

Now we are going to show that $B^3_{h0}$ admits a set of basis functions with vertex-patch-based supports.

\begin{theorem}\label{thm:strbh0}
The space $B^3_{h0}$ admits a set of basis functions each is supported in a patch of some vertex. 
\end{theorem}
\begin{proof}
By the exact sequence \eqref{eq:complexbh3}, we got to know that the piecewise gradient $\nabla_h$ is a bijection between $B^3_{h0}$ and $\uG{}^2_{h0}(\rot,0)$. Further, by Lemma \ref{eq:strucG2h0w0}, the set 
\begin{equation}\label{eq:basisbh0}
\{(\nabla_h)^{-1}\mathcal{F}_h\uphi{}^x_a,(\nabla_h)^{-1}\mathcal{F}_h\uphi{}^y_a,(\nabla_h)^{-1}\mathcal{F}_h\uphi{}_{P_a},(\nabla_h)^{-1}\mathcal{F}_h\uphi{}_e \}_{a\in\mathcal{X}^i_h,\ e\in\mathcal{E}^i_h}
\end{equation} 
form a basis of $B^3_{h0}$. Note again that both $(\nabla_h)^{-1}$ and $\mathcal{F}_h$ preserve the locality of the support; this is verified by viewing the patch as a specific triangulation. Namely \eqref{eq:basisbh0} is a basis each supported in the patch of a vertex. The proof is completed. 
\end{proof}

Similar to Theorem \ref{thm:strbh0}, we have the description below.
\begin{theorem}\label{thm:strbh0n}
The space $B^3_{ht}$ admits a set of basis functions each is supported in a patch of some vertex. 
\end{theorem}
\begin{proof}
By the complex \eqref{eq:complexbh3navier}, again, $\nabla_h$ is a bijection between $B^3_{ht}$ and $\uG{}^2_{ht}(\rot,0)$. By Lemma \ref{eq:strucG2htw0}, 
$$
B^3_{ht}=B^3_{h0}\oplus {\rm span}\{(\nabla_h)^{-1}\mathcal{F}_h\uphi{}_a^b:a\in \mathcal{X}_h^b\setminus\mathcal{X}^c_h\}\oplus{\rm span}\{(\nabla_h)^{-1}\mathcal{F}_h\uphi{}_e:e\in\mathcal{E}_h^b:\ e\in\mathcal{E}_h^b\},
$$
and a locally supported basis of $B^3_{ht}$ follows. The proof is completed.
\end{proof}

We use the notation below for convenience: 
\begin{equation}
\mbox{for}\ a\in\mathcal{X}_h^i\quad
w_{a}^x:=(\nabla^{-1})_h\circ\mathcal{F}_h\uphi{}_a^x,\
w_{a}^y:=(\nabla^{-1})_h\circ\mathcal{F}_h\uphi{}_a^y,\ 
w_{a}:=(\nabla^{-1})_h\circ\mathcal{F}_h\uphi{}_{P_a}^x,\ 
\end{equation}
\begin{equation}
\mbox{for}\ e\in\mathcal{E}_h,\quad  
w_e:=(\nabla^{-1})_h\circ\mathcal{F}_h\uphi{}_e;
\end{equation}
\begin{equation}
\mbox{and\  \ for}\ a\in\mathcal{X}^b_h\setminus\mathcal{X}_h^c,\quad \ w_a^b:=(\nabla^{-1})_h\circ\mathcal{F}_h\uphi{}_a^b.
\end{equation}

We remark here all these $w's$ can be obtained by straightforward calculation, as, again, both $(\nabla_h)^{-1}$ and $\mathcal{F}_h$ preserve the locality of the supports and can be done cell by cell. Though the space $B^3_h$ does not correspond to a finite element defined by Ciarlet's triple, these $w's$ play the same role as that by the usual nodal basis functions. Substituting these functions into the common routine generates finite element codes of the schemes \eqref{eq:bhvpp3} and \eqref{eq:bhvpp3n} in a standard way.

\subsection{Implementation by decomposition}

In this subsection, alternatively, we suggest a decomposition procedure, and the schemes \eqref{eq:bhvpp3} and \eqref{eq:bhvpp3n} can be implemented without the explicit construction of the basis functions. 

\begin{lemma}\label{lem:ncdecom}
Let $u_h^*$ be obtained by the following procedure: 
\begin{enumerate}
\item find $r_h\in A^3_{h0}$ such that 
\begin{equation}\label{eq:sub-1}
(\nabla_hr_h,\nabla_hs_h)=(f,s_h),\quad\forall\,s_h\in A^3_{h0};
\end{equation}
\item with $r_h$ obtained, find $(\uphi{}_h,p_h)\in \uG{}^2_{h0}\times\mathbb{P}^1_{h0}$ such that 
\begin{equation*}
\left\{
\begin{array}{ll}
(\nabla_h\uphi{}_h,\nabla_h\upsi{}_h)+(p_h,\rot_h\upsi{}_h)=(\nabla_hr_h,\upsi{}_h)&\forall\,\upsi{}_h\in\uG{}^2_{h0},
\\
(q_h,\rot_h\uphi{}_h)=0,&\forall\,q_h\in \mathbb{P}^1_{h0};
\end{array}
\right.
\end{equation*}
\item with $\uphi{}_h$ obtained, find $u_h^*\in A^3_{h0}$ such that 
$$
(\nabla_hu_h^*,\nabla_hv_h^*)=(\uphi{}_h,\nabla_hv_h^*),\quad\forall v_h^*\in A^3_{h0}.
$$
\end{enumerate}
Let $u_h$ be the solution of \eqref{eq:bhvpp3}. Then, $u_h^*=u_h$.
\end{lemma}

\begin{lemma}\label{lem:ncdecomn}
Let $z_h^*$ be obtained by the following procedure: 
\begin{enumerate}
\item find $r_h\in A^3_{h0}$ such that 
\begin{equation}%\label{eq:sub-1}
(\nabla_hr_h,\nabla_hs_h)=(f,s_h),\quad\forall\,s_h\in A^3_{h0};
\end{equation}
\item with $r_h$ obtained, find $(\uphi{}_h,p_h)\in \uG{}^2_{ht}\times\mathbb{P}^1_{h0}$ such that 
\begin{equation*}
\left\{
\begin{array}{ll}
(\nabla_h\uphi{}_h,\nabla_h\upsi{}_h)+(p_h,\rot_h\upsi{}_h)=(\nabla_hr_h,\upsi{}_h)&\forall\,\upsi{}_h\in\uG{}^2_{ht},
\\
(q_h,\rot_h\uphi{}_h)=0,&\forall\,q_h\in \mathbb{P}^1_{h0};
\end{array}
\right.
\end{equation*}
\item with $\uphi{}_h$ obtained, find $z_h^*\in A^3_{h0}$ such that 
$$
(\nabla_hz_h^*,\nabla_hv_h^*)=(\uphi{}_h,\nabla_hv_h^*),\quad\forall v_h^*\in A^3_{h0}.
$$
\end{enumerate}
Let $z_h$ be the solution of \eqref{eq:bhvpp3n}. Then, $z_h^*=z_h$.
\end{lemma}

Lemmas \ref{lem:ncdecom} and \ref{lem:ncdecomn} follows from Theorem \ref{thm:exactb3} and Lemma \ref{eq:imp3}. The scheme \eqref{eq:sub-1} is not a convergent one for the Poisson equation, but it is well-posed based on the continuity of $A^3_{h0}$ on vertices. With the formulations presented in Lemmas \ref{lem:ncdecom} and \ref{lem:ncdecomn}, the spaces used for Poisson equations and Stokes problems only are easy to formulate; to solve the system only needs solving two Poisson systems and one Stokes systems one by one, each of which can be solved with various optimal solvers in a friendly way.

\begin{remark}
The decompositions as in Lemmas \ref{lem:ncdecom} and \ref{lem:ncdecomn} can also be established for \eqref{eq:bhvpp3nd} and the one in Remark \ref{rem:laplapnavier}, respectively.
\end{remark}

\section{Conclusion and discussion}
\label{sec:conc}

In this paper, based on theoretical analysis by an indirect approach,  a constructive answer is given to the question if an optimal scheme can be designed for the biharmonic equation with piecewise cubic polynomials on general triangulations; the schemes work optimally, e.g., on triangulations shown in Figure \ref{fig:counter}. Beside the theoretical meaning, the scheme can find its application onto practical problems. For example, a high order scheme has been implemented based on $B^3_h$ for the Helmholtz transmission eigenvalue problem from inverse scattering and accoustics in \cite{Xi.Y;Ji.X;Zhang.S2019arxiv}; we refer there for many numerical experiments about schemes with $B^3_h$. The practical usage of the scheme can be thus illustrated.

This paper relies on construction and utilization of discretized Stokes complexes based on the $\uG{}^2_{h0}-\mathbb{P}^1_{h0}$ pair. The space $\uG{}^k_h$ with $k=3$ corresponds to the Crouzeix--Falk pair studied in \cite{Crouzeix.M;Falk.R1989}. In that paper, the authors proved that the pair $\uG{}^3_{h0}-\mathbb{P}^2_{h0}$ is stable ``for most reasonable meshes." Moreover, they presented a conjecture that the pair is stable ``for any triangulation of a convex polygon satisfying the minimal angle condition and containing an interior vertex." Recently, some triangulations where $\uG{}^3_{h0}-\mathbb{P}_{h0}^2$ is stable or at least $\dv\uG{}^3_{h0}=\mathbb{P}_{h0}^2$ are introduced in \cite{guzman2017cubic}. This hints the possibility to generalize the concept for optimal quartic element schemes (see \cite{Zhang.S2018} for details).

It is worthwhile pointing out, in this paper, we focus on the primal schemes only. There have been various kinds of schemes that considered new variables and/or conduct the second order differentiation in a dual way, such as the mixed element method, local DG method, hybridized DG method, CDG method, weak Galerkin method, and so forth. We remark that the literature on related works in this context is vast, but we will not discuss them in this paper. Moreover, based on the space $B^3_{h0}$($B^3_{ht}$), DG schemes can be designed. One may be able to construct, for example, a weakly over-penalized IP method (like \cite{Brenner.S;Gudi.T;Sung.L2010}) or IPDG method with optimal convergence rate robust with respect to the penalization paremeter (\cite{Zhang.S2019arxiv}) with piecewise cubic polynomials.

The spaces $A^3_h$ and $B^3_h$ each belongs to a systematic family which reads:
%\vskip-15mm
\begin{equation*}
A^k_h:=\{w_h\in L^2(\Omega):w_h|_T\in P_k(T); w_h(a)\ \mbox{is\ continuous\ at}\ a\in\mathcal{X}_h\}
\end{equation*}
and
\begin{multline*}
B_h^k:=\{w_h\in L^2(\Omega):w_h(a)\ \mbox{is\ continuous\ at}\ a\in\mathcal{X}_h;
\\ 
\fint_e\llt w_h\rrt p_e=0, \ \forall\,p_e\in P_{k-3}(e), \fint_eq_e\llt\partial_{\bf n}w_h\rrt=0,\ \forall\,p_e\in P_{k-2}(e),\ \forall\, e\in\mathcal{E}_h^i\}.
\end{multline*}
The spaces $A^k_{h0}$ and $B_{h0}^k$ can be defined corresponding to the boundary conditions of $H^1_0(\Omega)$ and $H^2_0(\Omega)$, respectively. It is now known that $B_{h(0)}^k$ is an optimally consistent finite element space for biharmonic equation ($k=2,3$) for arbitrary triangulations. For $k=4$, as discussed above, the assertion holds on most ``reasonable" triangulations. Can the family work optimally with arbitrary $k\geqslant 2$ and can it be generalized to a higher dimension and even higher-order problems? This question could be of interest in future research.

\end{document}